\documentclass[11pt,a4paper]{amsart}
\usepackage[T1]{fontenc}
\usepackage[utf8]{inputenc}
\usepackage{hyperref,amssymb,url,upref,verbatim,xspace,mathrsfs}
\RequirePackage[all,ps,cmtip]{xy}
\newdir^{ (}{{}*!/-5pt/\dir^{(}}
\newdir_{ (}{{}*!/-5pt/\dir_{(}}
\RequirePackage[mathscr]{eucal}
\usepackage{mathrsfs}\let\mathcal\mathscr
\usepackage[all]{xy}
\usepackage{graphicx}
{\usepackage[dvipsnames]{xcolor}

\newcounter{toto}
\def\thetoto{\arabic{toto}}
\let\oldmarginpar\marginpar
\def\marginpar#1{\refstepcounter{toto}\textsuperscript{\textup{[\thetoto]}}\oldmarginpar{\footnotesize\textsuperscript{[\thetoto]}\,#1}}

\addtolength{\topmargin}{-.5\baselineskip}

\makeatletter
\@namedef{subjclassname@2010}{\textup{2010} Mathematics Subject Classification}
\def\l@section{\@tocline{1}{0pt}{0pc}{}{}}
\def\l@subsection{\@tocline{2}{0pt}{1.5pc}{}{}}
\def\l@subsubsection{\@tocline{3}{0pt}{2pc}{}{}}
\makeatother

\newcommand{\C}{\mathbb{C}}\let\CC\C

\newcommand{\N}{\mathbb{N}}
\newcommand{\R}{\mathbb{R}}

\newcommand{\Z}{\mathbb{Z}}
\newcommand{\perv}{\mathrm{perv}}
\newcommand{\bD}{\boldsymbol{D}}

\newcommand{\Rhom}{R\shhom}
\newcommand{\shhom}{\mathcal{H}\!\mathit{om}}
\DeclareMathOperator{\rh}{\mathit{R}\shhom}
\DeclareMathOperator{\tho}{\mathit{T}\shhom}

\DeclareMathOperator{\RH}{RH}
\let\TH\THH

\newcommand{\rb}{\mathrm{b}}

\newcommand{\coh}{\mathrm{coh}}
\newcommand{\hol}{\mathrm{hol}}
\newcommand{\rhol}{\mathrm{rhol}}
\newcommand{\srhol}{\mathrm{srhol}}
\newcommand{\Mod}{\mathrm{Mod}}

\newcommand{\cc}{{\C\textup{-c}}}
\newcommand{\rc}{{\R\textup{-c}}}

\newcommand{\XS}{X\times S}

\newcommand{\DXS}{\shd_{\XS/S}}
\newcommand{\DXSa}{\shd_{\XS^*/S^*}}

\newcommand{\DYS}{\shd_{Y\times S/S}}

\DeclareMathOperator{\Char}{Char}

\DeclareMathOperator{\codim}{codim}

\DeclareMathOperator{\rD}{\mathsf{D}}

\DeclareMathOperator{\DR}{DR}
\DeclareMathOperator{\Db}{\mathfrak{Db}}

\DeclareMathOperator{\Hom}{Hom}
\DeclareMathOperator{\id}{Id}

\DeclareMathOperator{\Sol}{Sol}
\DeclareMathOperator{\pSol}{{}^\mathrm{p}Sol}
\DeclareMathOperator{\supp}{Supp}

\let\tilde\widetilde

\let\epsilon\varepsilon
\let\emptyset\varnothing
\let\setminus\smallsetminus
\let\leq\leqslant
\let\geq\geqslant

\def\loccit{loc.\kern3pt cit.{}\xspace}
\def\cf{cf.\kern.3em}
\def\eg{e.g.\kern.3em}

\def\resp{\text{resp.}\kern.3em}

\newcommand{\Di}{{}_{\scriptscriptstyle\mathrm{D}}i}

\newcommand{\pOS}{p^{-1}\sho_S}

\numberwithin{equation}{section}

\def\shc{\mathcal{C}}
\def\shd{\mathcal{D}}

\let\cF F
\let\cG G
\def\shh{\mathcal{H}}

\def\shj{\mathcal{J}}
\def\shh{\mathcal{H}}

\def\shl{\mathcal{L}}
\def\shm{\mathcal{M}}

\def\sho{\mathcal{O}}

 \def\shd{\mathcal{D}}
\let\cF\shf
\def\shh{\mathcal{H}}
\def\shj{\mathcal{J}}
\def\shl{\mathcal{L}}
\def\shm{\mathcal{M}}

\def\sho{\mathcal{O}}
\let\tilde\widetilde

\let\epsilon\varepsilon
\let\emptyset\varnothing
\let\setminus\smallsetminus
\let\leq\leqslant
\let\geq\geqslant

\def\loccit{loc.\kern3pt cit.{}X\times Space}
\def\cf{cf.\kern.3em}
\def\eg{e.g.\kern.3em}

\def\resp{\text{resp.}\kern.3em}

\newtheorem{theorem}{Theorem}[section]
\newtheorem{proposition}[theorem]{Proposition}
\newtheorem{lemma}[theorem]{Lemma}
\newtheorem{corollary}[theorem]{Corollary}

\theoremstyle{definition}
\newtheorem{definition}[theorem]{Definition}

\newtheorem{remark}[theorem]{Remark}

\newtheorem*{claim*}{Claim}

\newcommand{\RedefinitSymbole}[1]{%
\expandafter\let\csname old\string#1\endcsname=#1
\let#1=\relax
\newcommand{#1}{\csname old\string#1\endcsname\,}%
}
\RedefinitSymbole{\forall} \RedefinitSymbole{\exists}

\def\to{\mathchoice{\longrightarrow}{\rightarrow}{\rightarrow}{\rightarrow}}

\def\hto{\mathrel{\lhook\joinrel\to}}

\def\To#1{\mathchoice{\xrightarrow{\textstyle\kern4pt#1\kern3pt}}{\stackrel{#1}{\longrightarrow}}{}{}}

\let\oldbigoplus\bigoplus
\renewcommand{\bigoplus}{\mathop{\textstyle\oldbigoplus}\displaylimits}
\let\oldbigwedge\bigwedge
\renewcommand{\bigwedge}{\mathop{\textstyle\oldbigwedge}\displaylimits}
\let\oldbigcap\bigcap
\renewcommand{\bigcap}{\mathop{\textstyle\oldbigcap}\displaylimits}
\let\oldbigcup\bigcup
\renewcommand{\bigcup}{\mathop{\textstyle\oldbigcup}\displaylimits}

\begin{document}

\author{Luisa Fiorot and Teresa Monteiro Fernandes}
\title[Relative Riemann-Hilbert correspondence]{Relative strongly regular holonomic $\shd$-modules and the Riemann-Hilbert correspondence}

\date{\today}

\thanks{The research of L.Fiorot was  supported by project BIRD163492 "Categorical homological methods in the study of algebraic structures" and project DOR1749402. The research of T.Monteiro Fernandes was supported by supported by Funda\c c\~ao para a Ci\^encia e a Tecnologia, UID/MAT/04561/2013.}

\address{Luisa Fiorot\\ Dipartimento di Matematica ``Tullio Levi-Civita'' Universit\`a degli Studi di Padova\\
Via Trieste, 63
35121 Padova Italy\\ \texttt{luisa.fiorot@unipd.it}}

\address{Teresa Monteiro Fernandes\\ Centro de Matem\'atica e Aplica\c{c}\~{o}es Fundamentais-CIO and Departamento de Matem\' atica da Faculdade de Ci\^encias da Universidade de Lisboa, Bloco C6, Piso 2, Campo Grande, 1749-016, Lisboa
Portugal\\ \texttt{mtfernandes@fc.ul.pt}}



\begin{abstract}
We introduce the notion of strong regular holonomic $\DXS$-module and we prove that
the functor $\RH^S$ introduced in \cite{MFCS2} takes image in
$\rD^\rb_{\srhol}(\DXS)$ (complexes of $\DXS$-module whose cohomologies are
strongly regular).
We prove that for $\dim X=\dim S=1$ the functor solution functor $\pSol$ restricted to
  $\rD^\rb_{\srhol}(\DXS)$ is an equivalence of categories with quasi-inverse $\RH^S$.
\end{abstract}
\maketitle

\tableofcontents
\section*{Introduction.}

Let $X$ and $S$ be complex manifolds, with dimensions respectively $d_X$ and $d_S$. Let $p$ denote the projection $X\times S\to S$.

The main purpose of this paper is to clarify the notions of regularity for holonomic $\shd_{X\times S/S}$-modules, to introduce the notion of strong regularity and to explain the behaviour of the relative Riemann-Hilbert functor $\RH^S$ constructed in \cite{MFCS2} with respect to this new notion. 
More precisely, we start by giving a characterization of regular holonomic complexes when $d_S=1$  which was implicit in \cite{MFCS2} but not proved there as being equivalent to the previous one.
Supposing moreover $d_X=1$, we prove in Theorem~\ref{T1}
 that, generically in the sense of \cite{MFCS3}, that is, away of a discrete subset of $S$, a holonomic complex is regular if and only if, for any $x\in X$ the complex of holomorphic solutions restricted to $\{x\}\times S$ is isomorphic to the complex of solutions in the formal completion of  $\sho_{X\times S}$ along $\{x\}\times S$ and this last condition will be called strong regularity. This result is a relative version of the well known Kashiwara-Kawai's result in the absolute case (Theorem 6.4.1 of \cite{KK3}). 

For general $d_X$, replacing points (in the one dimensional case) by arbitrary hypersurfaces, leads us to the category 
$\rD^\rb_{\srhol}(\DXS)$ 
whose objects are complexes with strongly regular cohomologies (Definition~\ref{D:sr}). 
Our first main result is Theorem \ref{Tstrong} in which we prove that the functor
$\RH^S$ takes image in $\rD^\rb_{\srhol}(\DXS)$ (where $d_S=1$
while there is no assumption on $d_X$). 

We conjecture that the functor $\RH^S:\rD^\rb_{\cc}(p_X^{-1}\sho_S)\to \rD^\rb_{\srhol}(\DXS)$
is an equivalence of category with quasi-inverse $\pSol$.
As a first step in this direction we prove in
Proposition~\ref{Tequiv1} that 
the restriction of the solution functor  
$\RH^S:\rD^\rb_{\cc}(p_X^{-1}\sho_S)_t\to \rD^\rb_{\srhol}(\DXS)_t$
to torsion complexes
(those having support of the form in $X\times S_0$ where $S_0$ is a discrete subset of $S$)
 is an equivalence of categories. Also, in Proposition \ref{CNEW2}, the same holds true in the abelian category of modules of $D$-type in a general sense along a fixed normal crossing divisor.

As another positive result for our conjecture we prove in Theorem~\ref{Tequiv} that, if  $d_X=1=d_S$,
the functor $\RH^S:\rD^\rb_{\cc}(p_X^{-1}\sho_S)\to \rD^\rb_{\srhol}(\DXS)$ is indeed an equivalence 
improving the result obtained in \cite{MFCS3}. 

However our methods do not apply for $d_X>1$ because, among other features, although the functor $\RH^S$ behaves well under restriction to submanifolds, this is not true for arbitrary holomorphic morphisms.

\section{Regular holonomic $\DXS$-modules}

For a holomorphic function $f$ on $S$ we define $Li^*_f$ as being the derived functor 
$p^{-1}(\sho_S/\sho_S f)\overset {L}{\otimes}_{p^{-1}\sho_S}\cdot$ on the derived category of $p^{-1}\sho_S$-modules.
If $s$ is any point of $S$, we denote by $Li^*_s$, as in \cite{MFCS1} and \cite{MFCS2}, 
the derived functor $p^{-1}(\sho_S/m)\overset {L}{\otimes}_{p^{-1}\sho_S}\cdot$, 
where $m$ denotes the maximal ideal of holomorphic functions vanishing at $s$.

Hereafter, when we mention "torsion" we refer to the action of $p^{-1}\sho_S$. Recall that, when $d_S=1$, a torsion free (also called strict) module will be locally free over $p^{-1}\sho_S$.

Let us recall that given a triangulated category $\shc$, by Rickard's criterion (\cite{Nee}), 
a full triangulated category $\shc'$ of $\shc$ is a thick subcategory if and only if it is closed under direct factors in $\shc$
(which means that any direct summand of an object in $\shc'$ is in $\shc'$).
In our case the category $\shc=\rD^\rb_{\hol}(\shd_X)$ and we aim to study the thick subcategory of
regular holonomic complexes.

When the triangulated category $\shc$ is endowed with a bounded $t$-structure ${\rD}=({\rD}^{\leq 0},{\rD}^{\geq 0})$ one can require
that the subcategory $\shc'$ is compatible with the truncation functors of the $t$-structure ${\rD}$ i.e.
for any $\shm\in \shc'$ we have $\tau^{\leq 0}\shm, \tau^{\geq 1}\shm\in \shc'$. Due to the fact that any object
in $\shc$ has only a finite number of non zero cohomologies, the compatibility of $\shc'$  with the truncation functors of ${\rD}$
is equivalent to require that $\shh^i(\shm)\in \shc'$ for any $i\in \Bbb Z$. This condition is essential in order to proceed by
induction on the cohomological length of the complex.

In \cite{MFCS2} the following definitions were introduced:

\noindent $(\mathrm{Reg 1})$ 

\noindent i) A holonomic $\shd_{X\times S/S}$-module $\shm$ is regular if, for each $s\in S$, $Li^*_s\shm$ is a complex in $\rD^\rb_{\rhol}(\shd_X)$.

\noindent ii) A complex $\shm\in\rD^\rb_{\hol}(\shd_{X\times S/S})$ is regular holonomic if its cohomology groups $\shh^j(\shm)$ are regular holonomic.

An alternative and natural definition of regularity would be the following:

\noindent$(\mathrm{Reg 2})$ 

\noindent A complex $\shm\in\rD^\rb_{\hol}(\shd_{X\times S/S})$ is regular holonomic if, for each $s\in S$, $Li^*_s\shm\in\rD^\rb_{\rhol}(\shd_X)$.


 We will prove in Proposition~\ref{Prop:1} that
the previous definitions 
are equivalent for $d_S=1$ and that for any $S$ $(\mathrm{Reg 1})$ implies $(\mathrm{Reg 2})$.
We remark that
both definitions give thick triangulated subcategories of $\rD^\rb_{\hol}(\shd_{X\times S/S})$ and it is clear that
whenever  $\shm$ is concentrated in a single degree the conditions $(\mathrm{Reg 1})$ and $(\mathrm{Reg 2})$ are equivalent.
Condition $(\mathrm{Reg 1})$ is by definition compatible with the truncation functors and condition
$(\mathrm{Reg 2})$ is compatible with the truncation functors if and only if it is equivalent to $(\mathrm{Reg 1})$.
On the other side condition $(\mathrm{Reg 2})$ is compatible with base change on $S$ which means that
given $S'\stackrel{g}\to S$ a morphism of complex manifolds and $\shm\in \rD^\rb_{\hol}(\DXS)$ 
which satisfies $(\mathrm{Reg 2})$ we get that $L(\id_X\times g)^\ast (\shm)\in \rD^\rb_{\hol}(\shd_{X\times S'/S'})$
satisfies $(\mathrm{Reg 2})$ too.

\begin{remark}
Condition $(\mathrm{Reg 1})$ implies $(\mathrm{Reg 2})$.
To see this, we argue by induction on the length of $\shm$.
Without loss of generality, we may assume that $\shm\in \rD^{\geq 0}_{\hol}(\shd_{X\times S/S})$
and we consider the following distinguished triangle
$(A)\,\, \shh^0\shm\to\shm\to \tau^{\geq1}\shm\stackrel{+}\to $.
Let us assume that $\shm$ satisfies $(\mathrm{Reg 1})$, hence
by definition both $\shh^0\shm$ and $\tau^{\geq 1}\shm$ satisfy $(\mathrm{Reg 1})$.
As remarked, $\shh^0\shm$ satisfies $(\mathrm{Reg 2})$ too and by induction on the length of $\shm$, 
$\tau^{\geq 1}\shm$ satisfies $(\mathrm{Reg 2})$ which permits to conclude that $\shm$ also satisfies $(\mathrm{Reg 2})$.
\end{remark}

\begin{proposition}\label{Prop:1}
For $d_S=1$ condition $(\mathrm{Reg 1})$ is equivalent to $(\mathrm{Reg 2})$.
\end{proposition}
\begin{proof}
We shall argue by induction on the cohomological length of $\shm$. 
As above we may assume that $\shm\in \rD^{\geq 0}_{\hol}(\shd_{X\times S/S})$ and we consider the  distinguished triangle $(A)$.
Assume that $\shm$ satisfies $(\mathrm{Reg 2}$) and let $s_0\in S$. Taking a local coordinate on $S$ vanishing on $s_0$, we deduce an exact sequence 
$$0\to\shh^{-1}Li^*_{s_0}\shh^{0}\shm\to \shh^{-1} Li^*_{s_0}\shm\to \shh^{-1} Li^*_{s_0}\tau^{\geq 1}\shm (=0)$$ 
$$\to \shh^0 L i^*_{s_0}\shh^0\shm\to 
\shh^0 L i^*_{s_0}\shm\to \shh^0 L i^*_{s_0}\tau^{\geq 1}\shm\to 0$$ 
so that, for $k\geq 1$, 
$\shh^{k}Li^*_{s_0}(\shm)\simeq \shh^{k}Li^*_{s_0}(\tau^{\geq 1}\shm)$.
The category of regular $\shd_{X}$-modules is closed under subquotients, so we conclude that $\shh^{0}\shm$ satisfies $(\mathrm{Reg 2})$, hence $\tau^{\geq 1}\shm$ also satisfies $(\mathrm{Reg 2})$. Since $\shh^0\shm$ satisfies $(\mathrm{Reg 1})$, induction on the length  entails that $\shm$ also satisfies $(\mathrm{Reg 1})$.
\end{proof}

\begin{remark}
Despite the absolute case it is not clear if the category of 
relative regular holonomic $\DXS$-module is closed under subquotients 
in the category of holonomic $\DXS$-module 
even in the case of $d_S=1$. 
Let $0\to \shm_1\to \shm\to\shm_2\to 0$ be a short exact sequence of 
holonomic $\DXS$-module
such that the middle term $\shm$ is  regular holonomic.
Hence for any $s_0\in S$ ($d_S=1$) we obtain the long exact sequence:
$$
\
0\to \shh^{-1}Li^*_{s_0}\shm_1\to  \shh^{-1}Li^*_{s_0}\shm\to  \shh^{-1}Li^*_{s_0}\shm_2\to $$
$$\;\;\to
\shh^{0}Li^*_{s_0}\shm_1\to  \shh^{0}Li^*_{s_0}\shm\to  
\shh^{0}Li^*_{s_0}\shm_2\to  0
$$
and using the hypothesis $Li^*_s\shm\in\rD^\rb_{\rhol}(\shd_X)$
in general we can only conclude that 
$\shh^{-1}Li^*_{s_0}\shm_1$ and $\shh^{0}Li^*_{s_0}\shm_2$ belongs to
$\rD^\rb_{\rhol}(\shd_X)$.

We notice that if $\shm$ is a torsion regular holonomic $\DXS$-module
 both $\shm_1, \shm_2$ are  torsion regular holonomic too.

In the case of $\shm_1=t(\shm)$ and $\shm_2=f(\shm)$ respectively
the torsion sub-object  and the strict quotient of $\shm$
we have  $ \shh^{-1}Li^*_{s_0}\shm_2=0$
since $\shm_2$ is strict and hence any term of the previous long exact
sequence is regular (since by hypothesis $\shh^{-1}Li^*_{s_0}\shm$
and $\shh^{0}Li^*_{s_0}\shm$ are regular).
This permits to conclude that 
$t(\shm)$ and $f(\shm)$ are regular holonomic too.

If for any $d_S$ the category of regular holonomic $\DXS$-modules
would be closed by  subquotients one can prove by induction that
condition  $(\mathrm{Reg 1})$, is equivalent to $(\mathrm{Reg 2})$.
\end{remark}

\section{Strong regularity}

\subsection{Complementary results on $\DXS$-modules}

For any submanifold $Y\subset X$, one defines the formal completion of $\sho_{X\times S}$ along $Y\times S$, $\sho_{X\times S\widehat{|}Y\times S}$, by $$\sho_{X\times S\widehat{|}Y\times S}:=\varprojlim_{k\in\Z} \sho_{X\times S}/\shj^k$$ where $\shj$ denotes the ideal of holomorphic functions vanishing on $Y\times S$.

\begin{lemma}\label{L12}
For any subamanifold $Y\subset X$, there are functorial isomorphism on $\rD^\rb(\DXS)$
$$
(i)\Rhom_{\DXS}(R\Gamma_{[Y\times S]}(\shm), \sho_{X\times S})
\stackrel{\simeq}\to \Rhom_{\DXS}(\shm, \sho_{X\times S\widehat{|}Y\times S})$$
$$(ii) \Rhom_{\DXS}(\shm, \sho_{X\times S\widehat{|}Y\times S})|_{Y\times S}\simeq\Rhom_{\DYS}({}_Di^*_{Y\times S}\shm, \sho_{Y\times S}).$$
\end{lemma}
\begin{proof}
Since $R\Gamma_{[Y\times S]}(\shm)\simeq {}_D {i_{Y\times S}}_* {{}_D} i^*_{Y\times S}\shm$, $(ii)$ follows from$(i)$ by adjunction thanks to the relative version of \cite[Th. 4. 33]{Ka2}.

Let us now prove $(i)$:

We have $R\Gamma_{[Y\times S]}(\shm)\simeq  R\Gamma_{[Y\times S]}(\sho_{X\times S})\otimes^{L}_{\sho_{X\times S}} \shm$
hence
\begin{eqnarray*}
\Rhom_{\DXS}(R\Gamma_{[Y\times S]}(\shm), \sho_{X\times S}) \simeq  \hfill
\\
\Rhom_{\shd_{X\times S}}(\shd_{X\times S}\otimes_{\DXS}R\Gamma_{[Y\times S]}(\shm), \sho_{X\times S}) \simeq  \hfill
\\
\Rhom_{\shd_{X\times S}}(\shd_{X\times S}\otimes_{\DXS}\shm \otimes^{L}_{\sho_{X\times S}}R\Gamma_{[Y\times S]}(\sho_{X\times S}), \sho_{X\times S})\simeq  \hfill \\
\Rhom_{\shd_{X\times S}}(\shd_{X\times S}\otimes _{\DXS} \shm, \Rhom_{\sho_{X\times S}}
(R\Gamma_{[Y\times S]}(\sho_{X\times S}),\sho_{X\times S})\simeq  \hfill \\
 \Rhom_{\DXS}(\shm, \sho_{X\times S\widehat{|}Y\times S})\\
\end{eqnarray*}
where the last isomorphism follows from \cite[Prop. 2.2.2]{M}.
\end{proof}
%

\begin{proposition}\label{P:for}
For any submanifold $Y\subset X$ we have
 $$\,Li^*_s \sho_{X\times S\widehat{|}Y\times S}\simeq \sho_{X\widehat{|}Y}$$
\end{proposition}
 \begin{proof}
 Thanks to the properties of $Li^*_s$ (cf. \cite[Prop. 3.1]{MFCS1}) the result follows from Mittag-Leffler's condition since the morphisms $\sho_{X\times S}/\shj^{k+1}\to \sho_{X\times S}/\shj^k$ are surjective.
 \end{proof}
%

\begin{definition}\label{D:sr}
Let $\shm$ be a
holonomic $\DXS$-module; $\shm$ is called \textit{strongly regular along  an  hypersurface $Y\subset X$} if
%
the natural morphism 
$$(\ast\ast)\,\Rhom_{\DXS}(\shm, \sho_{X\times S})_{|Y\times S}\to \Rhom_{\DXS}(\shm, \sho_{X\times S\widehat{|}Y\times S})$$ is an isomorphism. 
%
%

$\shm$ 
is called \textit{strongly regular} if  it is strongly regular along any hypersurface.
Strongly regular holonomic $\DXS$-modules form a 
 full thick category $\Mod_{\srhol}(\DXS)$ of $\Mod_{\rhol}(\DXS)$ and we will denote by $\rD^\rb_{\srhol}(\DXS)$ the full subcategory of $\rD^\rb_{\hol}(\DXS)$
whose objects are complexes with strongly regular cohomologies.
%
\end{definition}

Thanks to Proposition~\ref{P:for} and Theorem 6.4.1 of \cite{KK3}
the condition of strongly regularity entails regularity in the sense of Definition ($\mathbf{Reg 1}$, $\mathbf{Reg 2}$).



\remark\label{R:asr}
Following \cite[Prop. II.2.2.3]{M} let $Y_1, Y_2$ be two hypersurfaces. The short exact sequence
$$0\to \sho_{X\times S\widehat{|}(Y_1\cup Y_2)\times S}\to \sho_{X\times S\widehat{|}Y_1 \times S}\oplus\sho_{X\times S\widehat{|}Y_2\times S} \to 
\sho_{X\times S\widehat{|}(Y_1\cap Y_2)\times S}\to 0$$
induces for any $\shm\in \rD^\rb_{\srhol}(\DXS)$ a distinguished triangle which permits to prove that
$$\Rhom_{\DXS}(\shm, \sho_{X\times S})_{|(Y_1\cap Y_2)\times S}\to \Rhom_{\DXS}(\shm, \sho_{X\times S\widehat{|}(Y_1\cap Y_2)\times S})$$ is an isomorphism too.  
Hence for any closed analytic subset $Z\subseteq X$ the natural morphism 
$$\Rhom_{\DXS}(\shm, \sho_{X\times S})_{|Z\times S}\to \Rhom_{\DXS}(\shm, \sho_{X\times S\widehat{|}Z\times S})$$ is an isomorphism
(since this is a local condition and we can reduce to a finite intersection of hypersurfaces).

\begin{remark}\label{R:ic}
Let $\shm\in \rD^\rb_{\hol}(\shd_{X\times S/S})$, let $Y$ be a closed analytic subset of $X$
and set $\tilde{Y}:=Y\times S$ for short.
By applying the solution functor to the distinguished triangle
$R\Gamma_{[\tilde{Y}]}(\shm)\to \shm \to \shm(\ast (\tilde{Y}))\stackrel{+}\to$
and according to the natural isomorphism $ \Rhom_{\DXS}(\shm, \sho_{X\times S\widehat{|}\tilde{Y}})\simeq
\Rhom_{\DXS}(R\Gamma_{[\tilde{Y}]}(\shm), \sho_{X\times S})$ 
we get the distinguished triangle
\begin{equation}\label{E:dt}\scalebox{0.83}{
\xymatrix@-15pt{
\Rhom_{\DXS}(\shm(\ast \tilde{Y}), \sho_{X\times S})_{|_{\tilde{Y}}}\ar[r] &
\Rhom_{\DXS}(\shm, \sho_{X\times S})_{|_{\tilde{Y}}}\ar[r] &
 \Rhom_{\DXS}(\shm, \sho_{X\times S\widehat{|}{\tilde{Y}}})_{|_{\tilde{Y}}} \\ }}
\end{equation}
which shows that $\shm$ is strongly regular along $Y$ if and only if
the complex
$\Rhom_{\DXS}(\shm(\ast (Y\times S)), \sho_{X\times S})_{|_{Y\times S}}=0$.
\end{remark}

\begin{proposition}\label{L:tor0}
Let $\shm\in \mathrm{Mod}_{\rhol}(\shd_{X\times S/S})$ be such that
 $\supp(M)\subseteq X\times T$ with $\dim T=0$. Then $\tilde{\shm}:=\shd_{X\times S}\otimes_{\DXS}\shm$ is a regular holonomic $\shd_{X\times S}$-module. 
\end{proposition} 

\begin{proof}

The statement being local, we may assume that $T=\{s_0\}$.  Hence $\Char(\shm)=\Lambda\times \{s_0\}$, where $\Lambda$ is a Lagrangian $\C^*$-conic  closed analytic subset  in $T^*X$, and, taking local coordinates $(z, s)$ in $X\times S$ such that $s$ vanishes in  $s_0$, there exists $n\in\N$ such that $(s-s_0)^n\shm=0$. Since we are dealing with triangulated categories, by an easy argument by induction on $n$  we may assume that  $n=1$. In that case, we have $\shm\simeq \shm_0\boxtimes \sho_S/\sho_S(s-s_0)$, where, by the assumption of relative regularity, $\shm_0$ is a regular holonomic $\shd_X$-module satisfying $\Char(\shm_0)=\Lambda$. By construction $\tilde{\shm}\simeq \shm_0\boxtimes \shd_S/\shd_S(s-s_0)$ and $\Char(\tilde{\shm})=\Lambda\times T^*_{T}S:=\tilde{\Lambda}$.

Therefore $\tilde{\shm}$ is a regular holonomic $\shd_{X\times S}$-module since the category of regular holonomic $\shd$-modules is closed under external tensor product.

\end{proof}

As an immediate consequence of loc. cit. \cite{KK3} we get:

\begin{corollary}\label{L:tor}

A complex $\shm\in \rD^\rb_{\rhol}(\shd_{X\times S/S})$ satisfying
 $\supp(M)\subseteq X\times T$ with $\dim T=0$ 
 is strongly regular, that is $\shm\in \rD^\rb_{\srhol}(\shd_{X\times S/S})$.
\end{corollary}

\subsection{Strong regularity for $d_X=1=d_S$}

 As defined in \cite{MFCS3}, a property in $X\times S$ is satisfied generically on $S$ if it is satisfied on $X\times S^*$, where $S^*$ is the complementary of a discrete subset $S_0$ in $S$.

We have the relative version of Theorem 6.4.1 of \cite{KK3}: 

\begin{theorem}\label{T1}
Let $\shm\in\rD^\rb_{\rhol}(\DXS)$. Then, generically on $S$, 
for any $x\in X$, the natural morphism 
$$\scalebox{0.94}{
\xymatrix@-12pt{
(\ast \shm): Rhom_{\DXS}(\shm, \sho_{X\times S})_{|_{\{x\}\times S^*}}\ar[r] & \Rhom_{{\shd_{X\times S^*/S^*}}}(\shm_{|_{X\times S^*}}, \sho_{X\times S^*\widehat{|}\{x\}\times S^*}) \\ }
}$$ 
is an isomorphism.
Conversely, if, for a given $\shm\in\rD^\rb_{\hol}(\DXS)$,  $(\ast \shm)$ is an isomorphism for each $x\in X$, with $S_0=\emptyset$, then $\shm\in\rD^\rb_{\rhol}(\DXS)$.
\end{theorem}
\begin{proof}
 a) Let us assume that $\shm\in\rD^\rb_{\rhol}(\DXS)$, 
 or equivalently $\shm\in\rD^\rb_{\hol}(\DXS)$ and for each $s\in S$, $Li^*_s\shm\in\rD^\rb_{\rhol}(\shd_X)$. 
 According to \cite[Th. 3.7]{MFCS1}, $\Rhom_{\DXS}(\shm, \sho_{X\times S})_{|\{x\}\times S}\in\rD^\rb_{\coh}(\sho_S)$, where we identify $\{x\}\times S$ with $S$. 
%
%

By \cite [Prop. 2.2 (4)]{MFCS3}, away of a discrete subset $S_0\subset S$, for any $x\in X$ 
 we have $R\Gamma_{[\{x\}\times S^*]}(\shm_{|_{X\times S^*}})\in\rD^\rb_{\hol}({\mathcal D}_{X\times S^*/S^*})$.
 Hence the morphism $(\ast \shm)$ is a morphism between complexes satisfying the finiteness condition of \cite[Prop 1.3]{MFCS2}. Therefore $(\ast \shm)$ will be an isomorphism provided that, for each $s\in S^*$, $Li^*_s(\ast \shm)$ is an isomorphism. 
 
 b) Let us now prove the converse. 
 If $(\ast \shm)$ is an isomorphism for each $x\in X$, then, after applying $Li^*_s$, in accordance with isomorphism $(\ast\ast)$ 
of Definition~\ref{D:sr}
 together with  \cite[Theorem 6.4.1]{KK3}, we conclude that, for each $s\in S$, $Li^*_s\shm$ is regular holonomic, hence $\shm\in\rD^\rb_{\rhol}(\DXS)$.
\end{proof}

\begin{corollary}\label{Lem:srii}
A complex $\shm$
in $\rD^\rb_{\rhol}(\DXS)$
is strongly regular if and only if it satisfies the equivalent conditions below:

i) $R\Gamma_{[\{x\}\times S]}(\shm) \in \rD^\rb_{\hol}(\DXS)$ for each $x\in X$

ii) 
${}_Di ^{\ast}_{\{x\}\times S}\shm\in \rD^\rb_{\coh}(\sho_S)$ for each $x\in X$.
\end{corollary}
\begin{proof}
 Let us prove first the equivalence $i)\Leftrightarrow ii)$. We have
$R\Gamma_{[\{x\}\times S]}(\shm)\simeq  {}_Di_{\{x\}\times S\; *}{}_Di^*_{\{x\}\times S}(\shm)[ d_X ]$
and hence $R\Gamma_{[\{x\}\times S]}(\shm) \in \rD^\rb_{\hol}(\DXS)$ if and only if
${}_Di ^{\ast}_{\{x\}\times S}\shm\in \rD^\rb_{\coh}(\sho_S)$.

Now, given $\shm\in\rD^\rb_{\srhol}(\DXS)$, in view of the definition we have a natural isomorphism
$$\Rhom_{\sho_S}({}_Di^*_{\{x\}\times S}\shm, \sho_S)\simeq\Rhom_{\DXS}(\shm, \sho_{X\times S})_{|_{\{x\}\times S}} \in  \rD^\rb_{\coh}(\sho_S)$$ which
proves that ${}_Di^*_{\{x\}\times S}\shm\in \rD^\rb_{\coh}(\sho_S)$ too. 

On the other side if ${}_Di ^{\ast}_{\{x\}\times S}\shm\in \rD^\rb_{\coh}(\sho_S)$ for any $x\in X$ we obtain that
$\Rhom_{\DXS}(\shm, \sho_{X\times S\widehat{|}\{x\}\times S}) \in \rD^\rb_{\coh}(\sho_S)$ and so
by \cite[Prop. 2.2]{MFCS1} the natural morphism 
$\Rhom_{\DXS}(\shm, \sho_{X\times S})|_{\{x\}\times S}\to \Rhom_{\DXS}(\shm, \sho_{X\times S^*\widehat{|}\{x\}\times S})$
is an isomorphism.
\end{proof}

\section{Application to the functor $\RH^S$}\label{sec:3}

In this section we briefly recall the relative Riemann-Hilbert functor $\RH^S(\cdot)$ introduced in \cite{MFCS2} and state some complementary results needed in the sequel.
We suppose $d_S=1$.
\subsection{Reminder on relative subanalytic sites and relative subanalytic sheaves}\label{subsec:relsubanalytic}
For details on this subject we refer to \cite{TL}. We also refer to \cite{KS5} as a foundational paper and to \cite{KS3} for a detailed exposition on the general theory of sheaves on sites.

Let $X$ and $S$ be real or complex analytic manifolds where we consider the family of open subanalytic subsets. On $X\times S$, $\mathcal{T}$ is the family consisting of finite unions of open relatively compact subsets and the family~$\mathcal{T}'$ consists of finite unions of open relatively compact sets of the form $U\times V$. The associated sites $({X\times S})_{\mathcal{T}}$ and $({X\times S})_{\mathcal{T}'}$ are nothing more than, respectively, $(X\times S)_{sa}$ and the product of sites $X_{sa}\times S_{sa}$.

We shall denote by $\rho$, without reference to $X\times S$ unless otherwise specified, the natural functor of sites $\rho:X\times S \to (X\times S)_{sa}$ associated to the inclusion $Op((X\times\times S)_{sa}) \subset Op(X\times S)$. Accordingly, we shall consider the associated functors $\rho_{*}, \rho^{-1}, \rho_!$.

We shall also denote by $\rho':X\times S \to (X\times S)_{\mathcal{T}'}$ the natural functor of sites. Following \cite{KS3} we have functors $\rho'_*$ and $\rho'_!$ from $\Mod(\CC_{X\times S})$ to $\Mod(\CC_{X_{sa}\times S_{sa}})$.

Subanalytic sheaves are defined on the subanalytic site of a real analytic manifold, and relative subanalytic sheaves are defined on the relative subanalytic site recalled above. We refer to \cite{TL}
for the detailed  construction of the relative subanalytic sheaves  $\shd^{t,S\sharp}_{\XS}$ (where $X$ and $S$ are real analytic) and $\sho^{t,S\sharp}_{\XS}$ (in the complex framework).

They are both $\rho_!\DXS$-modules (either in the real or the complex case) as well as a $\rho'_*p^{-1}\sho_S$-module and both structures commute.
Moreover, when $X$ is complex, considering the complex conjugate structure $\overline{X}$ on $X$ (resp. $\overline{S}$ on $S$) and the underlying real analytic structure $X_{\R}$ (resp. $S_{\R}$), 
we have $$\sho^{t,S\sharp}_{\XS}=\Rhom_{\rho'_!\shd_{\overline{X}\times\overline{S}}}(\rho'_!\sho_{\overline{X}\times\overline{S}},\shd^{t,S\sharp}_{\XS})$$
where we omit the reference to the real structures.
\subsection{Reminder on $\RH^S$ and complementary properties\label{subsec:RHS}}

In the real framework ($X$ and $S$ being real analytic manifolds, $d_S=1$), for $F\in\rD^\rb_\rc(p^{-1}\sho_S)$ we set

$$\TH^S(F):=
\rho'^{-1}\rh_{\rho'_*\pOS}(\rho'_*F, \Db^{t,S,\sharp}_{\XS}).$$

If $X$ is a complex manifold of complex dimension $d_X$ and $S$ is a complex manifold of dimension one,
$\RH^S:\rD^\rb_\rc(\pOS)\to\rD^\rb_\rhol(\DXS)$ is given by the assignment
$$F\to \RH^S(F):=
\rho'^{-1}\rh_{\rho'_*\pOS}(\rho'_*F, \sho^{t,S,\sharp}_{\XS})[d_X]$$ $$\simeq  \Rhom_{\shd_{\overline{X}\times\overline{S}}}(\sho_{\overline{X}\times\overline{S}}, \TH^S(F))$$
the last isomorphism being called here "realification procedure" for short.

\begin{proposition}\label{RHSloc}
Let $Y$ be a complex hypersurface of $X$. Then, for any $F\in\rD^\rb_\rc(\pOS)$ there is a natural morphism $$\RH^S(F)(\ast (Y\times S))\simeq \RH^S(F\otimes \C_{(X\setminus Y)\times S}).$$
In particular, if $F\in\rD^\rb_\cc(\pOS)$, 
\begin{itemize}
\item{1. $\RH^S(F)(\ast (Y\times S))$ belongs to $\rD^\rb_\rhol(\DXS)$,}
\item{2. There is a natural isomorphism $\RH^S(F\otimes \C_{Y\times S})\simeq R\Gamma_{[Y\times S]}(\RH^S(F))$ and so $R\Gamma_{[Y\times S]}(\RH^S(F))$ also belongs to $\rD^\rb_\rhol(\DXS)$.}
\end{itemize}
\end{proposition}

\begin{proof}
Let be given a local equation $f=0$ of $Y$. 
We start by assuming that $F=p^{-1}\sho_S\otimes\C_{\Omega\times S}$ for a relatively compact subanalytic open subset of $X$. Noting that $f$ is invertible on $ \tho(\C_{(\Omega\setminus Y)\times S}, \Db_{X\times S})$, according to \cite[Prop. 3.23]{Ka3},  the natural morphism

$$\tho(\C_{\Omega\times S}, \Db_{X\times S})(\ast (Y\times S))\to \tho(\C_{(\Omega\setminus Y)\times S}, \Db_{X\times S})$$ is an isomorphism. Since localization along $Y\times S$ is an exact functor and commutes with the realification procedure, according to \cite[Prop.3.5]{MFCS2} we conclude that $f$ is invertible on  $\RH^S(F\otimes \C_{(X\setminus Y)\times S})$ in $\rD^\rb_{\rhol}(\DXS)$, for any $F\in \rD^\rb_\cc(\pOS)$, which implies the existence of the morphism of  the statement.  Consequently it is an isomorphism.
The remaining statements follow straightforwardly (see also  \cite[Example 3.20]{MFCS2}).
\end{proof}

\begin{corollary}\label{C iminv}
For any  $F\in\rD^\rb_\cc(\pOS)$ and for any closed submanifold of $X$,  
$R\Gamma_{[Y\times S]}(\RH^S(F))$ is a complex with regular holonomic $\DXS$-cohomologies. Equivalently ${{}_D} i^*_{Y\times S}\RH^S(F)$ is a complex with regular holonomic $\DYS$-cohomologies.
\end{corollary}

\begin{proof}
The statement being local, we may assume that $Y$ is an intersection of smooth hypersurfaces of $X$ and then conclude by item $2$ of Proposition \ref{RHSloc} that 
$R\Gamma_{[Y\times S]}(\RH^S(F))\simeq \RH^S(F\otimes \C_{Y\times S})$ which implies the first statement.

According to the relative version of \cite[Prop.4.3]{Ka1}, the second statement is equivalent to the first. 
\end{proof}





\subsection{Relative Riemann-Hilbert correspondence and strong regularity}
The first main result in this section is the following:
\begin{theorem}\label{Tstrong}
For any $F\in\rD^\rb_{\cc}(p^{-1}\sho_S)$, $\RH^S(F)\in \rD^\rb_{\srhol}(\DXS)$.
\end{theorem}
\begin{proof}
Let $\shm:=\RH^S(F)$. We want to prove that, for any closed smooth hypersurface $Y$ of $X$, 
$$\Rhom_{\DXS}(\shm, \sho_{X\times S})_{|Y\times S}\to \Rhom_{\DXS}(\shm, \sho_{X\times S\widehat{|}Y\times S})$$ is an isomorphism. This amounts to prove that the right-hand side term is a $\C$-constructible complex since in that case, in view of Proposition \ref{P:for}, we can apply $Li^*_s$ for each $s\in S$ to conclude the result by reduction to the absolute case which holds true (cf. \cite{KK3}). According to Lemma \ref{L12} (ii) the $\C$-constructibility of  $\Rhom_{\DXS}(\shm, \sho_{X\times S\widehat{|}Y\times S})$ follows from Corollary \ref{C iminv} and from \cite[Th. 1.1]{MFCS1}.
\end{proof}

In \cite[Prop. 3.12]{FMF} the authors introduce the torsion class 
\[
\perv(p_X^{-1}\sho_S)_t:= \{F\in\perv(p_X^{-1}\sho_S) |\; \codim p_X(\supp F)\geq 1\}\hfill \\
\]

whose associated torsion-free class is denoted by  $\perv(p_X^{-1}\sho_S)_{tf}$.
The category $\perv(p_X^{-1}\sho_S)_t$ 
 is a full thick abelian subcategory of the category $\perv(p_X^{-1}\sho_S)$ of perverse 
sheaves. 
We denote by $\rD^\rb_{\cc}(p_X^{-1}\sho_S)_{t}$ the thick subcategory of 
$\rD^\rb_{\cc}(p_X^{-1}\sho_S)$ whose objects have support in $X\times S_0$ where $S_0$ is a discrete subset of $S$ or equivalently whose perverse cohomologies belong to 
$\perv(p_X^{-1}\sho_S)_t$.

In analogy we denote by 
$\rD^\rb_{\rhol}(\DXS)_t$  the thick subcategory of $\rD^\rb_{\rhol}(\DXS)$ 
whose objects have support in $X\times S_0$ where $S_0$ is a discrete subset of $S$.

\begin{proposition}\label{Tequiv1}
The restriction of the solution functor $\pSol$ to $\rD^\rb_{\rhol}(\DXS)_t$ is an equivalence of categories
$$\pSol:\rD^\rb_{\rhol}(\DXS)_t\to \rD^\rb_{\cc}(p_X^{-1}\sho_S)_t$$
with quasi-inverse the restriction of the functor
 $\RH^S$ to $\rD^\rb_{\cc}(p_X^{-1}\sho_S)_t$. 
 \end{proposition}
\begin{proof}
It will be sufficient to prove that the restriction of 
$\RH^S$ to $\rD^\rb_{\cc}(p_X^{-1}\sho_S)_{t }$ is fully faithful. 
Indeed $\Sol$ is essentially surjective since for any
 $F \in\rD^\rb_{\cc}(p_X^{-1}\sho_S)$ we have $F\simeq \pSol \RH^S(F) $
 and in the case of a torsion object $F$ in $\rD^\rb_{\cc}(p_X^{-1}\sho_S)_t$ we have
$\RH^S(F)\in \rD^\rb_{\rhol}(\DXS)_t$. 

For the full faithfulness it is enough to prove  that, for any $\shm\in \rD^\rb_{\rhol}(\DXS)_t$ and for any $G\in \rD^\rb_\rc(\pOS)$,  the morphism:

$$\Rhom_{\DXS}(\shm, \RH^S(G))\rightarrow
 \Rhom_{\DXS}(\shm, \Rhom_{\pOS}(G,\sho_{\XS})[d_X])$$ is an isomorphism. 
%

The cohomologies of $\shm$  are regular holonomic $\DXS$-modules satisfying the assumption of Corollary \ref{L:tor}. Hence $\shd_{X\times S}\otimes_{\DXS}\shm$ is a a complex with regular holonomic $\shd_{X\times S}$-modules as cohomologies.

Thanks to \cite[Prop.3.5]{MFCS2}, we may assume that 
$G=\C_{\Omega\times S}\otimes p_X^{-1}\sho_S$ 
 for some open subanalytic subset $\Omega$ of $X$, hence $\RH^S(G)=\tho(\C_{\Omega\times S}, \sho_{X\times S})[d_X]$ which is a complex with $\shd_{X\times S}$-modules as cohomologies and we get a chain of isomorphisms
 $$\Rhom_{\DXS}(\shm, \RH^S(G))\simeq \Rhom_{\shd_{X\times S}}(\shd_{X\times S}\otimes_{\DXS}\shm, \RH^S(G))$$ $$\underset{(\ast)}{\simeq}
 \Rhom_{\shd_{X\times S}}(\shd_{X\times S}\otimes_{\DXS}\shm, \Rhom(\C_{\Omega\times S},\sho_{\XS})[d_X])$$
 $$\simeq \Rhom_{\DXS}(\shm, \Rhom(\C_{\Omega\times S},\sho_{\XS})[d_X]) $$
 $$\simeq \Rhom_{\DXS}(\shm, \Rhom_{\pOS}(G,\sho_{\XS})[d_X])$$

 where isomorphism $(\ast)$ follows by \cite[Cor. 8.6]{Ka3}.

 \end{proof}
 
 \begin{corollary}\label{Ct}
 The category $\rD^\rb_{\rhol}(\DXS)_t$ is a triangulated subcategory of $\rD^\rb_{\srhol}(\DXS)$.
 \end{corollary}

 Recall the following definition:
\begin{definition}\label{D:D-tyep}(\cite[Definition 2.10]{MFCS2}
A coherent $\DXS$-module $\shl$ is said to be of $D$-type with singularities along a normal crossing divisor $D\subset X$  if it satisfies the following conditions:
\begin{enumerate}
\item $\Char(\shl)\subset(\pi^{-1}(D)\times S)\cup (T^*_XX\times S)$,
\item $\shl$ is regular holonomic and strict,
\item $\shl\simeq \shl(\ast (D\times S))$.
\end{enumerate}
\end{definition}

Let us fix a normal crossing divisor $D$.

Recall that in \cite[Lemma 4.2]{MFCS2} the authors proved that
the category of holonomic $\DXS$-modules $\shl$ of D-type along $D$ is equivalent to 
the  category of locally free $p_U^{-1}\sho_S$ with $U:=X\setminus  D$ (\cite[Prop. 2.11]{MFCS2})
under the  correspondence
$$
\shl\mapsto \shh^0\DR(\shl)_{|U\times S} \qquad F\mapsto \RH^S(j_! \bD(F[d_X]))=\shl$$
where $j:U\hookrightarrow X$ is the inclusion of the open $U$ in $X$.  
In particular there is a natural functorial isomorphism $\shl \stackrel{\simeq}\to \RH^S\pSol(\shl)$
for any holonomic $\DXS$-modules $\shl$ of D-type along $D$.
 Therefore we get:
 
 \begin{corollary}\label{Cnew}
Let $\shl$ be a holonomic $\DXS$-module of D-type along $D$. Then $\shl$ is strongly regular.
\end{corollary}

 
 Let $\mathbf{D}$ be 
 the full thick abelian subcategory  of $\Mod_{\rhol}(\DXS)$ whose objects  $\shl$ 
satisfy the conditions of $D$-type except for strictness:
\begin{enumerate}
\item $\Char(\shl)\subset(\pi^{-1}(D)\times S)\cup (T^*_XX\times S)$,
\item $\shl$ is regular holonomic,
\item $\shl\simeq \shl(\ast (D\times S))$.
\end{enumerate}
$\mathbf{D}$ is also endowed with a natural torsion pair $(\mathbf{D}_t, \mathbf{D}_{tf})$
induced as in the previous case  by the torsion on $S$, $\mathbf{D}_{tf}$ denoting the category of modules of $D$-type.


We have now the tools to conclude the following:
\begin{proposition}\label{CNEW2} For any object $\shm$ of $\mathbf{D}$ there is a canonical  isomorphism $\theta(\shm):\shm\to \RH^S(\pSol\shm)$ which is functorial in $\shm$.
In particular, any object $\shm$ of $\mathbf{D}$  is strongly regular. 
\end{proposition}
\begin{proof}

We shall prove  that, for any $\shm\in \mathbf{D}$ and for any $G\in \rD^\rb_\rc(\pOS)$,  the morphism:

$$(\ast)\,\,\Rhom_{\DXS}(\shm, \RH^S(G))\rightarrow
 \Rhom_{\DXS}(\shm, \Rhom_{\pOS}(G,\sho_{\XS})[d_X])$$ is an isomorphism. 
 
Then $\theta(\shm)$ will be the unique morphism in $\Hom_{\DXS}(\shm, \RH^S(\pSol\shm))$ corresponding to $Id\in \Hom_{p^{-1}\sho_S}(\pSol\shm, \pSol\shm)$. The fact that $\theta(\shm)$ is a functorial isomorphism is proved precisely as in \cite{MFCS3} so we will avoid the repetition of the proof.

According to the thickness of $\Mod_{\srhol}(\DXS)$ we may reduce the proof to the torsion case and to the torsion free case. In the torsion case, the result is contained in Proposition \ref{Tequiv1}.

If $\shm$ is torsion free, $\shm$ is of $D$-type along $D\times S$ and, according to the proof of \cite[Cor. 2.8]{MFCS2}, a devissage allows  us to consider local coordinates $(x,s)$ on $X\times S$ vanishing on $p=(x_0, s_0)\in D\times S$ such that $D=\{x\in X, x_1\cdots x_d=0\}$, and $\shm$ isomorphic to a quotient $\DXS/\sum_{i=1}^d\DXS(x_i\partial_{x_i}-\alpha_i(s))+\sum_{i=d+1}^n\DXS \partial _{x_i}$ for some holomorphic functions $\alpha_i(s)$, 
$i=1,\cdots,d$, on a fixed open neighborhood of $0$ in $S$.
We shall again reduce to the case $G=\C_{\Omega\times S}$ for some open subanalytic subset $\Omega$ of $X$. 
We have $$\Rhom_{\DXS}(\shm, \RH^S(G))\simeq \Rhom_{\DXS}(\shm, \Rhom_{\shd_{\overline{X}\times \overline{S}}}(\sho_{\overline{X}\times\overline{S}},\TH^S(G)[d_X])$$ $$\simeq\Rhom_{\shd_{X\times\overline{X}\times \overline{S}\times S/S}}(\shm\boxtimes \sho_{\overline{X}\times\overline{S}},\TH^S(G)[d_X])$$ $$\simeq \Rhom_{\shd_{X\times\overline{X}\times \overline{S}\times S/S}}(\shm\boxtimes \sho_{\overline{X}\times\overline{S}},\tho(\C_{\Omega\times S}, \Db_{X\times S}[d_X]))$$ $$\simeq \Rhom_{\shd_{X\times\overline{X}\times \overline{S}\times S}}((\shd_{X\times S}\otimes_{\DXS}\shm)\boxtimes \sho_{\overline{X}\times\overline{S}},\TH^S(G)[d_X]))$$ 

Remark that $$(\shd_{X\times S}\otimes_{\DXS}\shm)\boxtimes \sho_{\overline{X}\times\overline{S}}$$ is a fuchsian system along each hypersurface $D_j\times S:=\{(x,s), x_j=0\}$, $j=1,\cdots, d$ in the sense of \cite{Ta}. It is clear (cf. example 5.1 in \cite {TL}) that the solutions of the homogeneous system defining $\shm$ belong to $\tho(\C_{\Omega\times S}, \Db_{X\times S})$. The result is then an application of \cite[Th.1]{Ta} which entails the solvability  in $\Db_{0,0)}$ of the same system.

\end{proof}












 \subsection{The case $d_X=1$}
 The following result improves \cite{MFCS3} in the sense that we precise which categories are equivalent by means of $\RH^S$. However it remains conjectural that the condition of strong regularity is indeed not equivalent to that of regularity.
 
\begin{theorem}\label{Tequiv} 
The contravariant functor
$$\RH^S:\rD^\rb_{\cc}(p^{-1}\sho_S)\to \rD^\rb_{\srhol}(\DXS)$$ is an equivalence of categories and $\pSol$ is its quasi-inverse.
\end{theorem}
\begin{proof}
The statement will follow as in the proof of Proposition \ref{CNEW2} provided that for any $\shm\in\rD^\rb_{\srhol}(\DXS)$ and any $F\in\rD^\rb_\rc(\pOS)$ the natural morphism

 \begin{equation}\label{aaa}
\Rhom_{\DXS}(\shm, \RH^S(F))\rightarrow
 \Rhom_{\DXS}(\shm, \Rhom_{\pOS}(F,\sho_{\XS})[1])
 \end{equation}
  is an isomorphism. In \cite{MFCS3} the authors proved that 
for any $\shm\in\rD^\rb_{\rhol}(\DXS)$ and any $F\in\rD^\rb_\rc(\pOS)$, there exists a discrete $S_0\subset S$ depending on $\shm$ only such that 
\eqref{aaa}
is an isomorphism outside  $S_0$.
In  their proof (\cite[Prop. 2.5]{MFCS3}) the authors show that 
in the case $\dim\supp (\shm)\leq 1$ one has $S_0=\emptyset$ 
while for $\dim\supp (\shm)=2$ the set $S^\ast:=S\setminus S_0$
can be taken to be the biggest open subset of $S$ such that
 for any reduced divisor $Y$ of~$X$, setting $i_Y:Y\hto X$ the inclusion,
$\Di_{Y*}\,\Di_Y^*(\shm_*)$ has holonomic cohomologies as an object of the category 
$\rD^\rb(\DXSa)$. Hence, if $\shm\in\rD^\rb_{\srhol}(\DXS)$,
according to Corollary~\ref{Lem:srii} we obtain that $S_0=\emptyset$.
 \end{proof}

\bibliographystyle{amsalpha}
\bibliography{bibart}

\end{document}